\documentclass[12pt,a4paper]{article}

\usepackage[fleqn,reqno]{amsmath}
\usepackage[psamsfonts]{amsfonts}
\usepackage{amssymb,amsthm,mathrsfs}

\usepackage{graphicx}
\usepackage{epstopdf}
\DeclareMathOperator*{\sgn}{sgn}
\DeclareMathOperator*{\Arg}{Arg}

\newcommand{\dee}{{\mathrm d}}
\newcommand{\imag}{{\mathrm i}}
\newcommand{\eps}{\varepsilon}

\newcommand{\R}{\mathbb{R}}
\newcommand{\cyclereal}{\beta}
\newcommand{\cycleimag}{\alpha}
\newcommand{\rem}[1]{}

\newcommand{\rb}{{\bf r}}
\newcommand{\p}{{\bf p}}
\renewcommand{\L}{{\bf L}}
\newcommand{\e}{{\bf e}}

\newtheorem{theo}{Theorem}

\newtheorem{lem}[theo]{Lemma}

\title{Semi-global symplectic invariants of the \\spherical pendulum}

\author{
      H.~R. Dullin 
     School of Mathematics and Statistics\\
     The University of Sydney\\
     Sydney, NSW 2006, Australia\\
     {\tt Holger.Dullin@sydney.edu.au} 
}
\date{\today}
\begin{document}
\maketitle

\begin{abstract}
\vspace*{1ex}
\noindent
We explicitly compute the semi-global symplectic invariants near the focus-focus point of the 
spherical pendulum. A modified Birkhoff normal form procedure is presented to compute 
the expansion of the Hamiltonian near the unstable equilibrium point in Eliasson-variables. 
Combining this with explicit formulas for the action we find the semi-global symplectic
invariants near the focus-focus point introduced by Vu Ngoc \cite{VuNgoc03}.
We also show that the Birkhoff normal form is the inverse of a complete elliptic integral
over a vanishing cycle.
To our knowledge this is the first time that semi-global symplectic invariants near a focus-focus point
have been computed explicitly. We close with some remarks about the pendulum,
for which the invariants can be related to theta functions in a beautiful way.
\end{abstract}

\section{Introduction}

The spherical pendulum is a paradigm for an integrable system with a focus-focus point.
The classical treatment of the spherical pendulum can be traced to the beginnings of
mechanics in the work of e.g.\ Lagrange and Weierstra\ss, 
in particular because it appears as a limiting case of the Lagrange top.
Whittaker \cite[paragraph 55(ii)]{Whittaker37} gives the classical treatment, including the explicit solution in terms
of elliptic functions. Its reappearance in a modern context is related to the discovery 
of Hamiltonian Mondromy by Cushman and Duistermat 
\cite{Duistermaat80,CushDuist88}. This non-trivial global property 
of the actions and the rotation number of the spherical pendulum 
has not been noted in the classical treatment. 
The non-trivial topology of certain torus bundles  in the spherical pendulum is best
understood with the help of singular reduction, and the detailed
treatment of the spherical pendulum in this formalism is given in \cite[chapter IV]{CushmanBates97}.
Within the context of KAM theory the frequency map of the system was shown to be 
non-degenerate for all regular values of the energy-momentum map by Horozov \cite{Horozov90}
(also see \cite{Gavrilov00}).
Later in \cite{Horozov93} Horozov showed that the frequency ratio map does have singularities, 
so that the isoenergetic non-degeneracy condition of the KAM theorem fails on codimension 
one curves of non-critical values in the image of the energy-momentum map $F$.
The Hamiltonian monodromy has also been studied from a complex analytic point of view
in  \cite{BeukersCushman02} and \cite{Audin02}.

In view of these works the spherical pendulum is one of the best understood 
Liouville integrable systems. 
In this article we want to add to this knowledge the semi-global symplectic invariants 
of the focus-focus point of the spherical pendulum, and the expansion of its rotation number and period 
near this point in terms of the invariant and the Birkhoff normal form.

Well known  smooth invariants of a dynamical system are the eigenvalues at an equilibrium point, 
which are invariant under smooth coordinate transformations. 
For  Liouville integrable Hamiltonians the local symplectic classification of non-degenerate singularities
was given by Eliasson in \cite{Eliasson84}.
More global invariants have also been considered.
There are Bolsinov's orbital invariants \cite{Bolsinov95,BolFom04} which are invariant under smooth conjugacy.
Using these it was shown that the flow of the Euler top is $C^0$ orbitally equivalent to 
the geodesic flow on some ellipsoid \cite{BolFom04}. 
Moreover, this equivalence cannot be made $C^1$  \cite{BD96}.
Related results on $C^r$ equivalence of divergence free vector fields are given in \cite{Kruglikov99}. 
Restricting to smooth symplectic coordinate transformations the corresponding invariants 
(in one degree of freedom) were introduced by Dufour, Molino and Toulet \cite{Molino94}.
The semi-global symplectic invariants for the Liouville foliation of a Liouville integrable Hamiltonian 
system with two degrees of freedom in the neighbourhood of 
the separatrix of a simple focus-focus point were introduced by Vu Ngoc \cite{VuNgoc03},
and for hyperbolic-hyperbolic equilibria by Dullin and Vu Ngoc \cite{DVN05}. 
The latter paper also contains and explicit calculations of semi-global symplectic invariants near a hyperbolic-hyperbolic equilibrium.

Consider two Liouville integrable systems on the manifold $M^4$ with symplectic structure 
$\Omega$ and with energy-momentum maps $F_1, F_2 : M^4 \to \R^2$, 
each with a simple focus-focus point.
If these integrable systems have the same semi-global symplectic invariants at the focus-focus 
point then there is a symplectic map $\Psi : M^4 \to M^4$ in a full neighbourhood of the separatrix 
of the focus-focus point and 
a map $\Phi : \R^2 \to \R^2$ reparameterizing the integrals such that  
\begin{equation}  \label{eqn:equiv}
    F_1 \circ \Psi = \Phi \circ F_2\,.
\end{equation}
The invariants are given as the coefficients of the power series of the analytic part 
of the singular action expressed as a function of certain regular actions.
Coincidence of the invariants is a necessary condition that the systems can be globally 
equivalent in the sense of \eqref{eqn:equiv}. 
%
%
The coincidence of these invariants is also a necessary condition for the corresponding 
quantum systems to have the same eigenfunctions. The eigenvalues do not enter this description. 
The classical analogue of this statement is that the Hamiltonian or the second integral
can be changed with the function $\Phi$ without changing the invariants.
When considering a weaker notion of equivalence which does not allow the transformation 
$\Phi$ then the additional invariants are the coefficients of the (simultaneous) Birkhoff normal form of $F$.

The main objects in the classification are two different kinds of actions. 
First of all the standard canonical action with its singular behaviour 
near a separatrix is essential. But in addition also ``imaginary actions''
which are sometimes associated with the quantum mechanical tunnelling play a crucial role. 
In the spherical pendulum they are given as complete elliptic integrals over
cycles $\cyclereal$ and $\cycleimag$ that form a basis of cycles on the 
corresponding complex torus.
It turns out that the imaginary actions are intimately related to the Birkhoff normal form near
an unstable equilibrium point, more precisely they are the Eliasson coordinates defined in a 
neighbourhood of the equilibrium. The key idea for the construction of the semi-global 
symplectic invariants is to express the regularised canonical action
as a function of the imaginary action. In this way naturally a symplectic invariant is obtained
that is independent of the Hamiltonian.

In \cite{DVN03} the information contained in the transformation 
of the Hamiltonian into normal form \cite{Vey78,Eliasson84,Ruessmann87,Ito89,Zung05} 
combined with the semi-global symplectic invariants \cite{VuNgoc03}
where used to find the form of the frequency ratio map near a simple focus-focus point.
Here we apply the general theory to the spherical pendulum, and we show how
the two essential power series, the Hamiltonian in terms of Eliasson's coordinates
and the semi-global symplectic invariant contained in the action, can actually be computed
using a combination of adapted Lie-series normal form and classical analysis of elliptic integrals of the third kind. 
As a special case the semi-global symplectic invariants 
of the ordinary pendulum are found. We show that these invariants are given by the 
power series of Jacobi's nome in terms of the imaginary action. 

To our knowledge the present work gives the first example where the semi-global symplectic invariants near a focus-focus point are calculated explicitly for a particular system.
We hope that other systems will be added to the list, so that this will
lead the way to a classification of Liouville integrable systems.

\section{Spherical Pendulum}

The spherical pendulum is a point of mass $m$ constrained to move on a sphere under the influence 
of gravity. 
Let $\rb \in S^2 \subset \R^3$ be a point on the sphere $||\rb|| = l$ of radius $l$ 
where the norm is induced by the standard Euclidean scalar product $(\cdot, \cdot)$.
Instead of the usual (singular) spherical coordinates on the sphere here we start with 
a global description that is free of coordinate singularities.
With $-mg\e_z $ being the force of gravity, Newton's equations of motion are 
\begin{equation} \label{eqn:Newton}
    m \ddot \rb = -mg  \e_z  - \lambda m \rb
\end{equation}
and 
the Lagrange multiplier $\lambda$ is
\[
   \lambda = \frac{(\dot \rb, \dot \rb) - g (\rb, \e_z)}{(\rb, \rb)}
\]
determined such that $\dee^2 (\rb,\rb)/\dee t^2 = (\rb, \ddot \rb) + (\dot \rb, \dot \rb) = 0$. 
The initial condition must satisfy 
$\dee (\rb, \rb) / \dee t = 2 (\rb, \dot \rb) = 0$ in order to be contained in the tangent space of the sphere.
The system has global Hamiltonian (shifted by a constant so that the {\em upper} equilibrium has energy zero)
\begin{equation} 
   \label{eqn:GlobHam}
   H(\rb, \p)  = \frac{1}{2ml^2} (\L, \L) + mgl \left(  \frac{ (\rb, \e_z) } { ||\rb||} - 1\right)
\end{equation}
with momentum $\p = m \dot \rb$ conjugate to $\rb$, and angular momentum $\L = \rb \times \p$. 
This Hamiltonian with standard symplectic form $\dee \rb \wedge \dee \p$ has
$(\rb,\rb)$ and $(\rb,\p)$ as constants of motion, such that the
constraints $(\rb,\rb)=l^2$ and $(\rb,\p)=0$ are automatically preserved.
Using these constraints Hamilton's equations of motion become
\[
   \dot \rb = \frac{1}{ml^2} \L  \times \rb, \qquad
   \dot \p = \frac{1}{ml^2} \L \times \p  -  \frac{mgl}{||\rb||} \e_z +\frac{mgl}{||\rb||^3} ( \e_z, \rb) \rb
\]
and are equivalent to \eqref{eqn:Newton}.
The angular momentum $L_z = (\L, \e_z) = x p_y - y p_x$ is another constant of motion.
Fixing the constrains defines a system on
$T^*S^2$ and this system is Liouville integrable because the 
Poisson bracket of the remaining integrals vanishes, 
$\{ H, L_z \} = 0$, and 
$H$ and $L_z$ are independent almost everywhere.
Formally the Hamiltonian \eqref{eqn:GlobHam} 
also describes the Lagrange top  
written in the space fixed frame and with the body symmetry reduced, see, e.g. \cite{D03b}. 
The point $\rb$ is a point on the symmetry axis of the body, and $m l^2$ is the value
of the equal moments of inertia. However, for the Lagrange top $\L$ is introduced
as an {\em independent} variable (with the induced Poisson structure), and 
for the general motion $(\rb, \L) \not = 0$, which is impossible for the 
spherical pendulum.

The frequency of small oscillations around the stable lower equilibrium at $z = -l$ is $\nu = \sqrt{g/l}$.
Measuring length in units of $l$, mass in units of $m$, and time in units of 
$\sqrt{l/g}$ all constants can be removed from the Hamiltonian \eqref{eqn:GlobHam}.
Since the transformation to non-dimensionalised variables is not symplectic
we prefer to keep the constants for now.
In the following we are interested in the dynamics near the unstable upper equilibrium at $z = l$.

\section{Linear normal form} 

For the normal form calculation near the unstable equilibrium
a local symplectic coordinate system on $T^*S^2$ is needed. 
In the global Hamiltonian \eqref{eqn:GlobHam} we introduce 
local coordinates $(x,y,l)$ in $\R^3$ by $(x,y,z) = (x,y, \sqrt{l^2-x^2-y^2})$ 
near $z =  l$.
The conjugate momenta $(p_x, p_y, p_l)$ are found from the inverse transpose of 
the Jacobian of the transformation of the coordinates and with $p_l = 0$ the obvious result
that $p_x$ and $p_y$ are unchanged and $p_z = 0$ is found.
In these symplectic coordinates with 
the Hamiltonian reads
\[
H = \frac{1}{2ml^2}\left( p_x^2(l^2-x^2) +  p_y^2(l^2-y^2) -2 x y p_x p_y \right) +
       mg ( \sqrt{l^2-x^2-y^2} - l) \,.
\]
This Hamiltonian could have been derived directly from a Lagrangian using 
$x$ and $y$ as local coordinates.

Before transforming the quadratic part of $H$ into Williamson normal form \cite{Williamson36,Arnold78}
the symplectic scaling $x =   \xi/\sqrt{m\nu}$, $p_x =  p_\xi \sqrt{m\nu}$, 
similarly with $\eta, p_\eta$, is done.  Thus we have proved 
\begin{lem} \label{localsympcoo}
There exist local symplectic coordinates near the unstable equilibrium of the 
spherical pendulum with $\Omega = d\xi \wedge d p_\xi + d\eta \wedge d p_\eta$ such that 
\begin{equation} \label{eqn:lem1}
  H = \nu \left(  \frac12 (  p_\xi^2 + p_\eta^2)
    - \frac{\kappa}{2} (\xi p_\xi + \eta p_\eta)^2 +  \frac{1}{\kappa} ( \sqrt{1 - \kappa \rho^2 } - 1) \right)
\end{equation}
where $\rho^2 = \xi^2 + \eta^2$, $\nu^2 = g/l$ and $1/\kappa = ml^2 \nu = mgl/\nu$\,.
\end{lem}
Taylor expansion of the potential (the last term in \eqref{eqn:lem1}) gives
\[
    -  \frac12 \rho^2 
    - \frac{1}{8} \kappa \rho^4 
    - \frac{1}{16} \kappa^2 \rho^6
    - \frac{5 }{128} \kappa^3 \rho^8
    - \dots 
  = 
    -\sum_{n=1}^\infty  \frac{(2n-3)!!}{(2n)!!} \kappa^{n-1} \rho^{2n}  
     \,,
\]
so that $\kappa$ keeps track of the order of the non-linear terms.
By scaling all variables $\xi, \eta, p_\xi, p_\eta$ with $1/\sqrt{\kappa}$ 
(which is a symplectic transformation with multiplier $1/\kappa$) all constants could be
removed from the Hamiltonian up to the overall factor $\nu/\kappa$, which 
could then be removed by a linear scaling of time.
The two {\em fundamental constants} left in the Hamiltonian are $1/\nu$ with unit of time
and $1/\kappa$ with unit of angular momentum or action (i.e.\ ${\rm kg \, m^2 / s}$).
Without changing the symplectic structure and/or the time it is not possible to remove
them. The usual approach to non-dimensionalise all variables (including time)
is not used here until later because the semi-global symplectic invariants should allow comparison of 
different systems, and only in one of them these units could be chosen freely.

So far point transformations were used to simplify $H$. 
To achieve the Williamson normal form \cite{Williamson36} of the quadratic part of $H$
it turns out that a symplectic rotation in phase space is necessary.
The Hessian of the Hamiltonian has eigenvalues $\pm \nu$ with algebraic and geometric multiplicity two. Hence the equilibrium is a degenerate saddle-saddle point. 
Since there are no coupling terms each saddle can be treated separately in its canonical plane.
A rotation by $\pi/4$ brings the hyperbola $(p_\xi^2 - \xi^2)/2$ into the normal form $p_1 q_1$,
similarly for $(\eta, p_\eta)$.
Thus we have
\begin{lem}  \label{Williamson}
The Williamson normal form at the unstable equilibrium of the spherical pendulum 
is achieved by the linear symplectic transformation
\[
   \sqrt{2} \xi = q_1 - p_1, \quad
   \sqrt{2} p_\xi = q_1 + p_1, \qquad
   \sqrt{2} \eta = q_2 - p_2, \quad
   \sqrt{2} p_\eta = q_2 + p_2 \,. 
\]
The Hamiltonian in the new coordinates reads
\[
  H  = \nu \left( p_1 q_1 + p_2 q_2 - \kappa \frac18 (q^2 - p^2)^2 +  
     \frac{1}{\kappa} \sqrt{1 - \kappa \rho^2} + \frac{\rho^2}{2} - \frac{1}{\kappa} \right), 
\]
where $p^2 = p_1^2 + p_2^2$,  $q^2 = q_1^2 + q_2^2$
and $\rho^2 = p^2 /2+ q^2/2 - (p_1 q_1 + p_2 q_2)$.
\end{lem}
The quadratic part of the potential has been absorbed in the quadratic normal form terms
$H_2 = \nu( p_1 q_2 + p_2 q_2)$. 
The remaining terms of the potential are of order 4 and higher.
Together with the quartic term from the kinetic energy the quartic terms are
\[
   \frac{1}{\nu\kappa} H_4  =   - \frac18 (q^2-p^2)^2 - \frac{1}{32} (p^2 + q^2 - 2 (p_1 q_1 + p_2 q_2))^2 \,.
\]

The angular momentum in the new variables reads $q_1 p_2 - q_2 p_1$, 
but it does not enter the quadratic Williamson normal form. This is why the quadratic normal 
form gives a degenerate saddle, while from the point of view of the foliation of 
the integrable system it is a focus-focus point. The reason for this is that any 
function of the integrals $H$ and 
$L_z$ is again an integral, and thus has the same foliation. 
In particular the linear combination
$H + \omega L_z$ for any $\omega$ has an equilibrium of focus-focus type, 
with eigenvalues $\pm\nu \pm \imag \omega$.
Incidentally this also shows that this equilibrium is a non-degenerate singularity 
in the sense of Eliasson. 
In fact we have achieved more than the Williamson normal form, 
because the second integral is also simultaneously normalised. 
See Appendix~\ref{app:Linear} for a few more remarks about this.

\section{Non-linear normal form}

From the work of  \cite{Vey78,Eliasson84,Ruessmann87,Ito89} we know that 
there is a {\em convergent} Birkhoff normal form near a non-degenerate 
equilibrium of a Liouville integrable system.
We are now going to compute this normal form for the spherical pendulum.

The equilibrium is of focus-focus type and thus we know that the Hamiltonian 
can be expressed as a function of two commuting integrals $J_1$ and $J_2$, 
where 
\begin{equation} \label{eqn:J1andJ2}
   J_1 = q_1 p_1 + q_2 p_2, \quad 
   J_2 = q_1 p_2 - q_2 p_1 \,.
\end{equation}
At quadratic order the Williamson normal form has already achieved this, $H_2 = \nu J_1$.
But the higher order terms still contain other combinations of the variables
$q_1, q_2, p_1, p_2$. Using non-linear near-identity transformations the 
Hamiltonian can be normalized to arbitrary order. For this we need

\begin{lem}[``Integral-Angle'' coordinates] \label{integral-angle}
The integrals $J_1$ and $J_2$ as given by \eqref{eqn:J1andJ2} are momenta of a symplectic coordinate system
on $\R^4$ with the plane $q_1 = q_2 = 0$ removed. The canonically conjugate 
variables  are
\[
   \theta_1 = \ln \sqrt{q_1^2 + q_2^2}, \quad
   \theta_2 = \tan^{-1} \frac{q_2}{q_1}
\]
\end{lem}
where only $\theta_2$ is an angle.
\begin{proof}
The construction mimics the usual action-angle variables near an elliptic-elliptic equilibrium point, 
however, here we are near a focus-focus point.
Formally going into the complex plane we define 
$\zeta = q_1 +\imag q_2$, $\pi = p_1 + \imag p_2$ 
so that $J_1 - \imag J_2 = \zeta \bar \pi$ 
and $\theta_1 + \imag \theta_2 = \ln \zeta$. Then 
\[
    \dee (\theta_1 + \imag \theta_2) \wedge \dee ( J_1 - \imag J_2 )
   =\dee \ln \zeta \wedge  \dee \zeta\bar \pi= \dee \zeta \wedge  \dee \bar \pi 
\]
and taking the real part of the previous formula and using
$ \Re( \dee \zeta \wedge \dee \bar \pi) =  \dee q_1 \wedge \dee p_1 + \dee q_2 \wedge \dee p_2 $
gives the symplecticity of the transformation. The transformation is defined outside the set $q_1 = q_2 = 0$.
The inverse of the transformation is given by
\[ 
\begin{aligned}
    q_1 &= \cos\theta_2 \, e^{\theta_1}, \quad 
    & p_1 = (J_1\cos\theta_2 - J_2 \sin\theta_2) \, e^{-\theta_1}, \\
    q_2 &= \sin\theta_2 \, e^{\theta_1}, \quad
    & p_2 = (J_1\sin\theta_2 + J_2 \cos\theta_2) \, e^{-\theta_1} \,,
\end{aligned}
\]
so that $q^2 = e^{2 \theta_1}$ and $p^2 = e^{-2\theta_1}(J_1^2 + J_2^2)$.
\end{proof}

Inserting this change of coordinates in to the Hamiltonian of lemma~\ref{Williamson}
a new Hamiltonian is obtained which is a function of $J_1, \theta_1$, and $J_2^2$ only.

It is important to keep in mind that at this stage $J_1$ is only equal to 
the true Eliasson coordinates up to quadratic order. Normalising the Hamiltonian order by order 
each step of the normal form procedure introduces new variables $\tilde J_i, \tilde \theta_i$, and then the tildes are dropped.
In the limit of infinite order the Hamiltonian thus becomes a function of $J_1$ and $J_2$ only.
In our particular case $J_2$ is unchanged in the process, since $H$ is independent of $\theta_2$, 
in other words $J_2$ already is a global integral. Note that the conjugate angle $\theta_2$ 
is in general changed by the normalisation process.
We are now going to describe the details of a practical implementation of 
this iterative procedure.

The Hamiltonian is independent of $\theta_2$, because this is the angle conjugate 
to the global integral $L_z = J_2$. 
The identity $p^2 q^2 = J_1^2 + J_2^2$ (using the notation for $p$ and $q$ of lemma~\ref{Williamson})
allows to eliminate occurrences of $p^2$ in $H$ using $p = J \exp(-\theta_1)$
where $J^2 =  J_1^2 + J_2^2$.
Thus we find for the lowest order terms that are non-linear in $J_i$ 
\begin{equation}
\begin{aligned} \label{H4is}
  \frac1 {\nu\kappa} H_4 &= -\frac18 ( J^2 e^{-2\theta_1} - e^{2\theta_1})^2 - 
              \frac{1}{32}(  J^2 e^{-2\theta_1} + e^{2\theta_1} - 2 J_1)^2 \\
              &= -\frac{5}{32} J^4 e^{-4\theta_1} + \frac18 J_1 J^2 e^{-2\theta_1} 
               + \frac{1}{16}J_1^2 + \frac{3}{16}J_2^2
               + \frac18 J_1 e^{2\theta_1} - \frac{5}{32} e^{4\theta_1}
\end{aligned}
\end{equation}

Now we seek a perturbation theory that removes the $\theta_1$ dependence in this
and all higher order terms.
The classical approach with mixed variables generating functions 
is quite cumbersome at higher order; it is briefly described in Appendix~\ref{pert}. 
The better alternative is the Lie series approach, 
in which the (near identity) canonical transformations are generated by the time 
$\epsilon$ map of some flow. The Hamiltonian of this flow is the generating function
of the transformation which only depends on the old variables. The transformation is
obtained by integrating this flow, which is conveniently done by computing
iterated Poisson brackets in a kind of Baker-Campbell-Hausdorff formula.

The abstract formalism is well understood, see e.g.\ \cite{MeyerHall92}. 
A particular version of this type of perturbation 
theory has to specify a solution to the so called Lie-equation (or homological equation)
\[
    \{  H_2, W_i \} = H_i -  K_i
\]
where $H_i$ are the terms of order $i$ of the given Hamiltonian, 
and $K_i$ are the terms remaining in the transformed Hamiltonian,
and $W_i$ is the generator of the transformation, all at stage  $i$ of the iterated transformation process. 
The range ${\cal R}_i$ of the homological operator $\{ H_2, \cdot  \}$ is the set of terms that 
can be removed within the chosen class of transformations generated by $W_i$.
A particular normal form consists of choosing a complement to the range ${\cal S}_i$,
which is the set of unremovable terms at stage $i$. 
Hence $K_i \in {\cal S}_i$ and $H_I - K_i \in {\cal R}_i$.
If the homological operator $\{  H_2, \cdot  \}$ is semisimple 
${\cal S}_i$ is the kernel of the homological operator.
In the classical example of Birkhoff normal form near an elliptic 
equilibrium point $K_i$ is the average of $H_i$ over the flow of $H_2$, 
while $R_i$ is the oscillating part. The solvability condition for $W_i$ 
is that $K_i - H_i$ has zero average.

We assume that the Hamiltonian is  in Williamson normal form of lemma~\ref{Williamson}
and written in terms of the ``integral-angle'' coordinates of lemma~\ref{integral-angle}, 
hence $H_2 = \nu J_1$ for the spherical pendulum.
A higher order term $H_i$ is in normal form when it is a function of $(J_1, J_2)$ alone. 
This is similar to the usual Birkhoff normal form near an elliptic equilibrium. 
The essential difference is that here $J_1$ is not a canonical action (it does not have
a periodic flow, in other words, $\theta_1$ is not an angle).
Define a family of classes of functions ${\cal P}_{i}$ as functions 
of the form
\[
    \sum_{k=0}^i Q_k(J_1, J_2) e^{(i - 2k)\theta_1}
\]
where $Q_k$ is a degree $k$ homogeneous polynomial.
The operator $\{ H_2, \cdot \} : {\cal P}_i \to {\cal P}_i$ simply gives
$W_i \mapsto \nu \partial W_i/ \partial \theta_1$. This operator is simple and 
splits into kernel and range. The kernel consists of functions ${\cal S}_i$ that are independent 
of  $\theta_1$. Because of the particular form of the class of functions ${\cal P}_i$ 
only the identity is in the kernel for odd $i$, while for even $i$ 
the kernel consists of polynomials $Q_{i/2}$.
The range consists of functions ${\cal R}_i$ where each term does depend on $\theta_1$, 
in other words ${\cal P}_i$ with ${\cal S}_i$ removed.

Using this in conjunction with theorem VII.B.4 in \cite{MeyerHall92} we can now prove
\begin{lem} \label{nofo}
The Lie series normal form with respect to $\{ H_2, \cdot \} $ for the spherical pendulum is unique
and the normalised Hamiltonian is a function of $J_1$ and $J_2^2$ only.
\end{lem}
\begin{proof}
We have already observed that this operator acting on ${\cal P}_i$ is simple.
For uniqueness we need to check $\{ {\cal P}_i, {\cal P}_j \} \subset {\cal P}_{i+j}$
and $\{ {\cal S}_i, {\cal S}_j \}  = 0$. The latter is true since $J_1$ and $J_2$ commute.
The former holds because the commutator yields a function in ${\cal P}_{i+j-1} \subset {\cal P}_{i+j}$.
The transformation given in lemma~\ref{integral-angle} combined with
lemma~\ref{Williamson} shows that all higher order terms $H_i$ are in ${\cal P}_{i}$.
Finally we need to solve the Lie equation, and show that it has a solution in ${\cal R}_i$. In our
case the Lie equation reads
\[
   H_i - K_i = \nu\frac{ \partial W_i}{\partial \theta_1} \,.
\]
Since the left hand side is in ${\cal R}_i$ (indefinite) integration with respect to $\theta_1$ 
gives $W_i \in {\cal R}_i$.
The normalised Hamiltonian is a function of $J_2^2$ only because the original Hamiltonian 
has the discrete symmetry $J_2 \to -J_2$ as well. Thus in fact the polynomials $Q_k$ are also 
even functions of $J_2$, and this property is preserved in the normalisation process.
\end{proof}

Note that there are no small denominators in this process, since there is just one angle involved.
Moser showed in \cite{Moser58} that even in the general non-integrable case 
a loxodromic equilibrium point has a convergent normal form.
The above theorem shows how the Eliasson normal form of the spherical pendulum
can be computed using Lie series. In the resulting normalised Hamiltonian all the higher 
order terms are in ${\cal S}_i$, i.e.\ they depend on $(J_1, J_2)$ only,
as desired. Thus we obtain

\begin{theo} \label{BNF1}
The normal form at the focus-focus point of the spherical pendulum 
is 
\begin{align} 
\label{hofj}
  &  \frac{1}{\nu} H(J_1, J_2)  = 
    J_1 + \frac{\kappa}{16}( J_1^2 +  3 J_2^2) - \frac{\kappa^2}{256}J_1(J_1^2 + 9 J_2^2) + \\
    & + \frac{\kappa^3}{8192}(5J_1^4 + 102 J_1^2 J_2^2 + 33 J_2^4)  
      - \frac{3 \kappa^4}{262144} J_1( 11 J_1^4 + 410 J_1^2 J_2^2 + 271 J_2^4) + O(J^6) \,.
\nonumber
\end{align}
\end{theo}
\begin{proof}
The proof is obtained by straightforward but lengthy calculation along the lines given in 
lemma~\ref{nofo}, and the details are omitted. Here we just give the first 
steps of this computation to illustrate the process.
There are no third order terms, so $H_3 = K_3 = W_3 = 0$. 
Next the $\theta_1$ independent terms in $H_4$ can be read off from \eqref{H4is}:
\[
   \frac{1}{\kappa\nu}K_4 = \frac{1}{16}J_1^2 + \frac{3}{16} J_2^2
\]
and integration of the remaining terms gives the generating function whose flow removes these terms:
\[
  \frac{1}{\kappa} W_4(J, \theta) = \frac{1}{\nu\kappa } \int (H_4 - K_4)d\theta_1 = 
             \frac{5}{128} J^4 e^{-4\theta_1} - \frac{1}{16} J_1 J^2 e^{-2\theta_1} 
                    + \frac{1}{16} J_1 e^{2\theta_1} - \frac{5}{128} e^{4\theta_1} \,.
\]
The terms of third order in $J$ are the unremovable terms in $H_6 + \frac12 \{ H_4, W_4 \}$.
\end{proof}

{\bf Remarks:} 
1) Setting $J_2 = 0$ the normal form of the ordinary pendulum at its unstable 
equilibrium point is obtained.
2) It appears that by the non-symplectic scaling $J_i \to 32 J_i$ all coefficients of $H(J_1, J_2)$ become integers.
3) The normal form procedure works for any integrable system with a focus-focus point that has a global integral $J_2$, 
so that the higher order terms are independent of $\theta_2$.

\section{Action Integrals}

The normal form calculation of the previous section is local, and hence only 
valid in a neighbourhood of the equilibrium point. The local normal form actions 
are denoted by $J_i$, while the global canonical actions are denoted by $I_i$.
In the present case the integral of angular momentum is $J_2 = I_2$. 
The semi-global theory is 
valid in a neighbourhood of the pinched torus which is the separatrix of the
equilibrium point. In order to find the semi-global  canonical action $I_1$
(which in this case is actually global up to monodromy)
we introduce spherical coordinates $(\phi, \theta, l)$ by the point transformation
$(x,y,z) = (l\cos\phi\sin\theta, l\sin\phi\sin\theta,l \cos\theta)$.
Here the inclination angle $\theta$ is measured down from the north pole.
The resulting Hamiltonian is 
\begin{equation} \label{eqn:Ham}
H =\frac12 \kappa\nu \left( p_\theta^2 +  p_\phi^2 \frac{1}{\sin\theta^2} \right) + \frac{\nu}{\kappa} (\cos\theta-1) \,,
\end{equation}
where $\theta = 0$ corresponds to the maximum of the potential which has energy zero.
Here $p_\phi = L_z = J_2$. This is the most compact form of the Hamiltonian, which
unfortunately has a singularity at $\theta = 0$. 
The singularity can be avoided using singular reduction, see
Cushman and Bates \cite[chapter IV]{CushmanBates97}.
The final action integral is identical in both approaches.
The action integral is obtained  by integrating $p_\theta \dee \theta$ over the real $\cyclereal$ cycle, 
hence
\[
\begin{aligned}
\kappa I_1(h, j_2) & = \frac{1}{2 \pi} \oint_\cyclereal \sqrt{\frac{2\kappa h}{\nu} - 2 (\cos\theta-1) - \frac{\kappa^2 j_2^2}{\sin^2\theta}} \dee \theta  \\
 & = \frac{1}{2\pi }\oint_\cyclereal \sqrt{2(\tilde h -  \zeta+1)(1-\zeta^2) - \tilde j_2^2} \frac{\dee \zeta}{1-\zeta^2} = \tilde I_1(\tilde h, \tilde j_2) \,,
\end{aligned}
\]
where the dimensionless length $\zeta = \cos\theta$ and the abbreviations
$\tilde I_1 = \kappa I_1$, $\tilde j_2 = \kappa j_2$ and $\tilde h = \kappa h/\nu$ for
the dimensionless actions and energy were introduced.
\footnote{These tildes are dropped in the following!}
The non-trivial action $I_1$ of the spherical pendulum is thus defined on the elliptic curve
\[
     \Gamma_{h,j_2} = \{ (\zeta, w) : w^2 = P(\zeta)  \}, \quad 
     P(\zeta) = 2(1-\zeta^2)( h+1-\zeta) -  j_2^2 \,.
\]
Denote the roots of $P$ by $\zeta_i$ where $-1 \le \zeta_0 \le \zeta_1 \le 1 \le \zeta_2$. 
The elliptic integrals needed in the following have been given in terms
of Legendre's standard integrals in e.g.~\cite{RDWW96}.
The modulus and parameters for all of them are
\[
     k^2 = \frac{\zeta_1 - \zeta_0}{\zeta_2 - \zeta_0}, \quad
     n_\pm = \frac{\zeta_1 - \zeta_0}{\pm 1 - \zeta_0} \,.
\]
With these definitions the action $I_1$ can be written as
\begin{align*}
   I_1( h,  j_2) 
    & = \frac{1}{2\pi} \oint_\cyclereal \frac{w}{1-\zeta^2} \, \dee \zeta  \\
    & = c_0( c_1 K(k) + c_2 E(k) + c_{3+} \Pi(n_+, k) + c_{3-} \Pi(n_-, k)  ) 
\end{align*}
where 
$
        c_0 = \frac{4}{\pi \sqrt{2(\zeta_2 - \zeta_0)}}, 
        c_1 =  1 + h - \zeta_2, 
        c_2 = \zeta_2 - \zeta_0, 
        c_{3+} =   \frac{j_2^2}{4(1 - \zeta_0)}, 
         c_{3-} =   \frac{j_2^2}{4(1 + \zeta_0)}  
$.
The $\cyclereal$-cycle encircles the the interval $[\zeta_0, \zeta_1]$ along which $w^2 \ge 0$.

\begin{lem} \label{lem:Iseries}
The action of the spherical pendulum near the unstable equilibrium as a function 
of $h$, the energy difference to the critical value, and the angular momentum $j_2$ has the expansion
\begin{equation} \label{Iseries}
\begin{aligned}
  2\pi I_1(h, j_2) & = 8 - 2 \pi | j_2 | + j_2 \Arg(h + \imag j_2 ) + J_1(h,j_2) \ln\frac{32}{\sqrt{h^2+j_2^2}} + \\
 & \quad  +  h + \frac{3}{32}( h^2 + 3 j_2^2) - \frac{h}{256(h^2+j_2^2)}(6h^4 + 43j_2^2h^2 + 39 j_2^4) + \dots
\end{aligned}
\end{equation}
The coefficient in front of the logarithm (the ``imaginary action'') is
\[
 J_1(h,j_2) = h - \frac{1}{16}(h^2 + 3j_2^2) + \frac{3}{256}h(h^2+5j_2^2) - \frac{5}{8192}(5h^4+42 j_2^2h^2 + 21j_2^4) + \dots \,.
\]
\end{lem}
\begin{proof}
For the computation of the expansion of $I_1$ the complete elliptic integrals need to be expanded 
in the singular limit that the modulus $k \to 1$. We set $h \to h \eps$ and 
$j_2 \to j_2 \eps$ and expand in $\eps$. The roots of $P(\zeta)$
in this limit are 
\begin{align*}
   \zeta_0 & = -1 + \frac18 j_2^2 + O(\eps^3) \\
   \zeta_1 & = 1 + \frac12 \left(h - \sqrt{h^2+j_2^2} \right) \left( 1 - \frac{j_2^2}{8\sqrt{h^2+j_2^2}} \right) + O(\eps^3) \\
   \zeta_2 & = 1 + \frac12 \left(h + \sqrt{h^2+j_2^2} \right)  \left( 1 + \frac{j_2^2}{8\sqrt{h^2+j_2^2}} \right)  + O(\eps^3) 
\end{align*}
so that the modulus  becomes
\[
   k^2 = 1 - \frac12 \sqrt{h^2 + j_2^2} + O(\eps^2)\,,
\]
and the parameters are 
\[
   n_+ = 1 + \frac14 \left(h - \sqrt{h^2 + j_2^2}\right) + O(\eps^2), \quad
   n_- = -\frac{16}{j_2^2} + O(1/\eps)\,.
\]
The fact that $n_-$ behaves like $1/j_2^2$ does not cause problems because
$c_{3-}$ has a factor $j_2^2$. 
For the expansion in the present so called `circular' case 
the integrals of third kind need 
to be rewritten in terms of 
Heuman's $\Lambda_0$ function, see e.g.\ \cite{BF71}.
Since there are two integrals of third kind some simplification is achieved when
they are combined using the addition law for $\Lambda_0$, again see \cite{BF71}.
The result is 
\[
    I_1 =   c_0 \tilde c_1 K(k)+ c_0 c_2 E(k) - \frac{j_2}{2} \Lambda_0( \varphi, k), 
\]
where 
\[
   \varphi = \pi - \arcsin\sqrt{\frac{j_2^2/2 - \zeta_1- \zeta_0}{\zeta_2 - \zeta_1}}, \quad
    \tilde c_1 = h+1-\zeta_2 - \frac{j_2^2}{4(1+\zeta_2)} 
       - \frac{j_2}{2} \sqrt{\frac{(\zeta_2-1)(\zeta_2-\zeta_1)}{2(\zeta_2+1)}}\sin\varphi \,.
\]
From this point onwards the expansion of $I_1$ 
to high order is rather lengthy but straightforward using the formulae
given in \cite{AS,BF71}. For a few more details of a similar calculation in a different system 
with a focus-focus point see \cite{DI03b}.

\end{proof}

The non-smooth term $-2\pi |j_2|$ in \eqref{Iseries} is present in order to 
make the numerical values coincide with the definition as a real integral, 
along with the choice of the principal argument $\Arg$.
In particular $I_1(h,j_2)$ is an even function of $j_2$, whose $j_2$-derivative 
limits to 1 for $h > 0$ and to $1/2$ for $h < 0$.
Since the spherical pendulum has monodromy, integer multiples of $j_2$ 
can be added to $I_1$ in order to make it locally smooth. When $I_1$ is made
smooth in this way it becomes globally multivalued. 
Making use of the freedom to add integer multiples of $j_2$ 
the term could also be omitted.

The function $J_1$ (which multiplies the singular $\ln$-term in $I_1$) is also a complete elliptic integral over the same differential $p_\theta \dee \theta$, 
but along the $\cycleimag$-cycle that is vanishing in the limit approaching the equilibrium. 
The $\cycleimag$-cycle is a closed cycle in the original $\theta$ variables. 
When the integral is rationalised by $z = \cos\theta$ this cycle is folded once over itself, 
so that in the $\zeta$-variables it must be multiplied by 2.
This factor of two is absorbed in our definition of $\cycleimag$, 
which traverses the cycle twice in the $\zeta$ variables. 
With this convention we have
\[
    J_1(h,j_2) = \frac{1}{2\pi \imag} \oint_\cycleimag \frac{w}{1-\zeta^2} \, \dee \zeta 
\]
where $\cycleimag$ goes twice around the interval $[\zeta_1, \zeta_2]$ along which $w^2 \le 0$.
In the limit $\epsilon \to 0$ these roots coalesce at 1.
Since $w$ is pure imaginary along this path we multiply by $-\imag$ to make $J_1$ real
(by abuse of notation nevertheless we call $J_1$ the imaginary action).

As a real integral $J_1$ is divergent, but as a complex integral it is well 
defined, and it can be transformed to Legendre normal form.
The real expression in this so called `hyperbolic' case can best be expressed 
in terms of Jacobi's Zeta function \cite{AS}.

Instead of using the complicated Legendre normal form, 
for $J_1$ the series expansion can be found more easily from 
an expansion of the integrand. In contrast to the integral $I_1$ an expansion works for $J_1$
because the integration cycle is vanishing in the expansion limit.
The residue of the differential form $\omega = w/(1-\zeta^2) \dee \zeta$ at $\zeta = \pm 1$ is $\mp |j_2|/2$.
When $h = j_2 = 0$ the polynomial factors as $P(\zeta) = 2(\zeta+1)(\zeta-1)^2$ so that
in this limit the pole of $\omega$ at $\zeta = 1$ cancels. 
When $h$ and $j_2$ are both small the differential $\omega$ can be expanded
first into a Taylor series in $\epsilon$ after setting $h \to h \epsilon$ 
and $l \to l \epsilon$, and second in a Laurent series at $\zeta = 1$. 
Each term is polynomial in $h, j_2$ and rational in $\zeta, \sqrt{\zeta+1}$ 
so that the residue at $\zeta=1$ of each term can be easily computed. 
Combining the two expansions an explicit formula for the coefficient of terms of 
degree $n$ in $h, j_2$ is found to be
\begin{equation} \label{resform}
   2  \frac{1}{(n-1)!} \frac{\partial^{n-1}}{\partial \zeta^{n-1}} \left(\left.   
          \frac{(\zeta-1)^{n}}{n!}\left. \frac{\partial^n}{\partial \epsilon^n} \frac{w}{1-\zeta^2} \right|_{\epsilon = 0}
    \right)\right|_{\zeta=1} \,.
\end{equation}
This gives the degree $n$ term in the series of $J_1(h, j_2)$ as stated in lemma~\ref{lem:Iseries}.
The factor of two appears because $\cycleimag$ is traversed twice in the $\zeta$ variables.

In a semiclassical context $J_1/2$ may have the interpretation of a tunnelling integral, which is used in uniform WKB quantisation across a separatrix.
For a tunnelling integral the integration path in the original variable $\theta$ is not closed, 
but goes from one side of the energetically allowed region to the other.
For our imaginary action $J_1$ the cycle is closed in the original variable $\theta$, 
thus the factor of $1/2$ appears when considering the tunnelling integral.
The interpretation of the imaginary action as a tunnelling integrals does hold 
for the ordinary pendulum, which essentially is obtained from the spherical pendulum
by setting $j_2=0$ (see section~\ref{pendulum}).
For the spherical pendulum, or in general for a system with a focus-focus point, however, the uniform semiclassical quantisation is more complicated, see \cite{Child98,VuNgoc00}.
In \cite{RDWW96} a semiclassical treatment of the spherical pendulum with an added cardanic frame
was considered. In its uniform quantisation tunnelling integrals are used, 
however, they diverge in the limit of vanishing frame.
Thus the analogy of the imaginary action (the integral over the $\cycleimag$-cycle) 
with the tunnelling integral is rather metaphorical for the spherical pendulum.

From the classical point of view the interesting observation is that $J_1(h, j_2)$ is the inverse of 
$H(j_1, j_2) = h$ with respect to $j_1$, so that $H(J_1(h,j_2),j_2) = h$ is an identity. 
%
Thus we have

\begin{theo} \label{BNFinvJ}
The Birkhoff normal form near the focus-focus point of the spherical pendulum
as given in Theorem~\ref{BNF1}
is also given by the inverse of the elliptic integral $J_1(h,j_2)$ (over the $\alpha$-cycle which is vanishing near the 
focus-focus equilibrium point) with respect to its first argument.
\end{theo}
\begin{proof}
After reduction, the spherical pendulum with one degree of freedom $\theta$ is given by \eqref{eqn:Ham}.
The dependence on the parameter $j_2$ is ignored in the following.
Recall that the transformation to (ordinary) action angle variables near an elliptic equilibrium point in one degree of freedom 
can be found from the mixed variables generating function $S(I, q) = \int^q p \, \dee q$, where $p$ is obtained by solving $H(p,q) = h(I)$ for $q$. 
This generating function is fixed by the requirement $\oint \dee \phi = 2 \pi$, 
and this gives $\oint p \, \dee q / ( 2\pi ) = I$ where the closed loop integrals are performed over the 
real solution of the equation $H(q,p) = h$ (see, e.g.~\cite{Arnold78}).
Thus the series expansion of the inverse of $I(h)$ gives the Birkhoff normal form at the elliptic equilibrium point.

The only modification we have to make to adapt this construction to our case of 
a hyperbolic equilibrium point is to allow a closed loop integral over a
complex solution of the equation $H(q,p) = h$. In particular consider a Hamiltonian of the form $p^2/2 + V(q)$ and 
consider a solution for which $p$ is purely imaginary, 
i.e. the real integration range in $q$ is an interval where no classical motion is allowed. 
The corresponding normalisation then is $\oint \dee \phi = 2 \pi \imag$, and 
accordingly $\oint p \, \dee q / ( 2 \pi \imag) = J$. The variable $\phi$ that is conjugate to $J$ is 
defined by $\phi = \partial S/\partial J$ as before. 
Performing this construction at the quadratic hyperbolic equilibrium $H = \frac12( a^2 p^2 - b^2 q^2)$
reproduces the transformation to $(J_1, \theta_1)$ as described in lemma~\ref{integral-angle}.
In the general case we have constructed $J_1(h)$, whose inverse gives the Birkhoff normal form 
at the hyperbolic equilibrium point.

Finally there is nothing special about a closed non-contractible loop that has real $q$ and imaginary $p$;
in the complex plane we can deform this path without changing $J$. In particular when the energy 
is above the energy of the saddle point of the Hamiltonian a path that encircles critical points with 
imaginary $q$ and real $p$ also works.
\end{proof}

The main advantage of using this theorem for the computation of the Birkhoff normal form 
instead of the Lie-series approach is that it is much simpler to first compute $J_1(h, j_2)$ 
from the local expansion of the corresponding integral using the residue expansion \eqref{resform}, 
and then find the Birkhoff normal form by inversion of the series.

\section{Semi-global Symplectic Invariants}

From the general theory \cite{VuNgoc03} it is known that the non-trivial action
near a non-degenerate focus-focus point has the following form
\[  
\begin{aligned}
  2\pi I_1(j_1, j_2) & = 2\pi I_{10} - \Re( \hat j \ln \hat j - \hat j ) + S(j_1, j_2)   
\end{aligned}
\] 
where $I_{10}$ is a constant and $\hat j = j_1 + \imag j_2$.
Choosing a different branch of the complex logarithm changes $I_1$ by an integer multiple
of $j_2$, and hence is a manifestation of the Hamiltonian monodromy 
of the focus-focus point. In the following theorem two different choices
of the branch cut are made depending on the sign of $j_2$, so that 
the result is consistent with \eqref{Iseries}.
%
%
\begin{theo} \label{theo:invariants}
The non-trivial canonical action $I_1$ of the sphercial pendulum near the unstable focus-focus point as a function 
of $j_1, j_2$  is given by
\begin{equation}  \label{eqn:Action}
2\pi  I_1(j_1, j_2)  = 8 - 2\pi  |j_2| +  j_2 \Arg \hat j - j_1 \ln |\hat j| +  j_1 + S(j_1, j_2)
 \end{equation}
 where the semi-global symplectic invariant $ S(j_1, j_2) $  is given by 
 \[
 S(j_1, j_2) =  j_1 \ln 32 + \frac{3}{32}(j_1^2 + 3j_2^2) 
    - \frac{1}{512}j_1 ( 5j_1^2 + 51j_2^2)
    + \frac{1}{32768} ( 55 j_1^4 + 1230 j_1^2 j_2^2 + 271 j_2^4) + \dots
\]
\end{theo}
\begin{proof}
We have already noticed in Theorem~\ref{BNFinvJ} that $J_1(h, j_2)$ is the inverse of the Birkhoff normal 
form, and thus inserting the normal form \eqref{hofj} into the series expansion of 
the action \eqref{Iseries} produces the term $-j_1 \ln | \hat j|$, as desired.
The miracle in this computation is that the coefficients of the regular part of $I_1(h, j_2)$ as obtained
from the expansion of the elliptic integral are {\em rational} functions of $h$ and $j_2$, 
see \eqref{Iseries},
but upon substitution of the Birkhoff normal form \eqref{hofj} they turn into {\em polynomials}
in $j_1$ and $j_2$. The general theory \cite{VuNgoc03} predicts this result,
but it seems to be hard to see this otherwise.
The resulting regular power-series part $S(j_1, j_2)$ of the series $I(j_1, j_2)$ is the semi-global symplectic 
invariant of the spherical pendulum at its focus-focus point.
\end{proof}

{\bf Remarks:} 
1) Setting $j_2 = 0$ the semi-global symplectic invariants of the ordinary pendulum  at its unstable  equilibrium point are obtained.
2) As noted in lemma~\ref{nofo}, $H$ is a function of $j_2^2$ only, and the same action $I_1$ is also even in $j_2$.
Thus the discrete symmetry of the Hamiltonian becomes a discrete symmetry of $S$.
3) The coefficients of $S$ (except for $\ln 32$) appear to become integers when each momentum is
scaled by 32. The fact that there are rational coefficients appears to be related to 
the fact that the spherical pendulum is analytic.
4) 
The maximal absolute error between the action integral $I_1(h, j_2)$ and the approximation $I_1(j_1, j_2)$ given in Theorem~\ref{theo:invariants}
together with the formula for $J_1(h, j_2)$ of lemma~\ref{lem:Iseries} using only the terms displayed in these formulas is less than
0.0001 inside a circle of radius $1/2$ centered at the critical point $(h,l) = (0,0)$, 
and less than 0.0032 inside a circle of radius $1$.
5) In this and the previous section we are using scaled quantities which absorb the constants $\nu$ and $\kappa$. 
Unscaling shows that $\nu$ does not appear in the symplectic invariant, while $1/\kappa$ with units of an action naturally does appear
in order to give $S$ the correct units of an action. But the unscaling also gives the term $j_1 \ln \hat j \kappa$ in $I_1$, and one could 
argue that expanding the $\ln$ gives another contribution to $S$ as $j_1 \ln \kappa$. However, the argument of 
$\ln$ should be dimensionless, and thus one should leave the $\kappa$ inside, and do not incorporate it into the symplectic invariant.

In \cite{DVN03} it was shown that the rotation number has the form
\[ 
   2\pi W(j_1, j_2) = -\arg \hat j - A \ln |\hat j| + A S_1 - S_2 \bmod 2 \pi \,,
\] 
where $S_i = \partial S/\partial j_i$. 
The function $A$ can be computed from the Birkhoff normal form of $H$ given 
in \eqref{hofj} as
$A = \partial_{j_2} H/ \partial_{j_1} H$.
Notice that for an elliptic equilibrium point where both $j_i$ are true actions 
the function $A$ would be the rotation number itself. In the present case it is a
major ingredient of the rotation number, but the rotation number also depends on the
foliation via the semi-global symplectic invariant $S$, and it has leading order 
singular terms which are universal for focus-focus points.
After expansion of the quotient the result for $A$ is
 \begin{equation} \label{Aofj}
 A(j_1, j_2) = \frac38 j_2 - \frac{15}{128}j_1 j_2 + \frac{15}{1024} j_2( 3j_1^2 + 2j_2^2) 
 - \frac{45}{65536} j_1 j_2 (25 j_1^2 + 43 j_2^2) + \dots 
 \end{equation}
In order to obtain the rotation number given by the 
complete elliptic integral obtained from $-\partial I_1(h, j_2)/ \partial j_2$ 
the necessary choice of branch cut gives
\[
   2\pi W(j_1, j_2) = - \frac{ \partial I_1(h, j_2)}{ \partial j_2 } = 2\pi  \sgn j_2  - \Arg \hat j - A \ln |\hat j| + A S_1 - S_2 \,.
\]
The additional term preserves the property that $W$ is odd in $j_2$ and 
 that the limiting values on the axis $j_2 = 0$ 
are $\pm 1$ for $j_1 > 0$ 
and $\pm 1/2$ for $j_1 < 0$, 
since  for $j_2 = 0$ we have $A = S_2 = 0$.
This gives a very compact formula for the rotation number of the spherical pendulum
near the focus-focus point.
\footnote{ In Cushman and Bates \cite{CushmanBates97} the rotation number $W$ is an even function
of $j_2$, which is obtained from our odd $W$ by taking the absolute value. Thus $W$ becomes continuous
and odd.
}

The negative rotation number is obtained from the action by partial differentiation with respect to $j_2$ keeping $h=const$. 
Differentiating \eqref{Iseries} we see that the function $A$ multiplying the $\ln$-term in the rotation number 
is $\partial J_1(h,j_2) / \partial j_2$. 
This shows that $A$ is the integral over the same differential as $W$, 
but instead over the vanishing cycle. 
Hence $A$ is a complete elliptic integral as well.
Moreover, this shows that the series of $A(h,j_2)$ can be found by differentiating
the series of $J_1(h, j_2)$ with respect to $j_2$.
Another even simpler pair of functions related to $(I_1, J_1)$ are the periods of motion obtained 
by differentiating $(I_1(h,j_2), J_1(h,j_2))$ with respect to $h$, instead of with respect to $j_2$. 
In all cases the series expansions of the integrals 
over the vanishing cycle can be easily found by expanding and computing 
residues as described for $J_1$.

\section{Dynamical Interpretation}

The physical interpretation of the rotation number is easily understood when 
considering a rational rotation number $W = p/q$. This means that the orbit
completes $q$ periods in the inclination angle $\theta$, while it completes 
$p$ periods of the azimuthal angle $\phi$ around the pole.

From a dynamical point of view the most interesting quantities that have simple interpretations
are derivatives of the action $I_1(h,j_2)$ which give the period of the (symmetry reduced) motion 
and the rotation number. The period is found to be 
\[
T(j_1, j_2) = 
	\frac{ \partial I_1/\partial j_1 }{\partial H/\partial j_1}
             =  \frac{ -\ln|\hat j| + S_1}{\partial H/ \partial j_1} = \frac{ \ln \frac{32}{|\hat j|} + \frac{3}{16} j_1 + O(2) }{ 1 + \frac18 j_1 + O(2)} 
\]
where $S_i = \partial S/\partial j_i$. 
The period diverges logarithmically when approaching the unstable equilibrium point at $\hat j = 0$.
The $j_1$ derivative of the symplectic invariant $S$ gives the smooth correction to this singular behaviour, 
when the factor $1/(\partial H/\partial j_1)$ multiplying the singular term $\ln |\hat j|$ is factored out.
The original construction in \cite{VuNgoc03} achieves the splitting into singular and smooth contribution
using two Poincar\'e sections transverse to the stable respectively unstable manifold of the equilibrium.
This divides the separatrix into an inner and an outer part, where the inner part approaches the 
equilibrium point for $|\hat j| \to 0$.
This gives an interpretation to the coefficient $\ln 32$ of $j_1$ in $S$, which is the leading order term 
of the time it takes to traverse the outer part of the separatrix.

In the plane of $(j_2, j_1)$ we now introduce polar coordinates by $j_2 = r \cos s$, $j_1 = r \sin s$.
The roots become 
$\zeta_0 = -1 + \frac18 r^2 \cos^2 s  + O(r^3)$ and
$\zeta_{1,2} = 1 + \frac12(1 \pm \sin s)(\pm r)( 1 \pm \frac{1}{16} r ) + O(r^3)$.
Recalling that $\zeta = \cos\theta$ shows that the minimal distance 
of the orbit to the north pole 
measured in the inclination angle $\theta$ 
is $ \sqrt{r(1- \sin s)}  + O(r^{3/2})$, 
while the minimal distance to the south pole is $ r|\cos s|/2 + O(r^2)$.
For $s \to \pi/2$ the angular momentum $j_2$ vanishes and both distances vanish.
For $s \to -\pi/2$ the angular momentum also vanishes, but since $j_1 \approx h < 0$
the north pole is still inaccessible. 
In figure~\ref{Wrat} the radius of the excluded region in the centre of each picture is $\sqrt{r(1- \sin s)}/2$ to 
leading order (the factor 1/2 comes from the stereographic projection), while the overall size is
$4/( r |\cos s|)$ to leading order.

It is interesting to note that to leading order the modulus 
$k^2 = 1 - r/2 + \frac{1}{32}( 4 + 3 \sin s) r^2 + O(r^3)$ is independent of $s$.
The leading order of the rotation number $W$ for $j_2 \ge 0$ in polar coordinates is
\[
  W = \frac34 + \frac{s}{2\pi} + \frac{3r  \cos s}{32 \pi}  \, ( 2 \ln \frac{32}{r} - 3  ) + O(r^2) \,.
\]
This clearly shows that along the half-circle $s \in [-\pi/2, \pi/2]$ in the right half-plane with $j_2 \ge 0$
the rotation number changes from $W = 1/2$ to $W = 1$, so by continuity every rotation number in between occurs at least once on 
each such half-circle. 
At least for small $r$ every rotation number occurs {\em exactly} once along the half-circle with radius $r$.

\begin{figure}
\centerline{ \includegraphics[width=15cm]{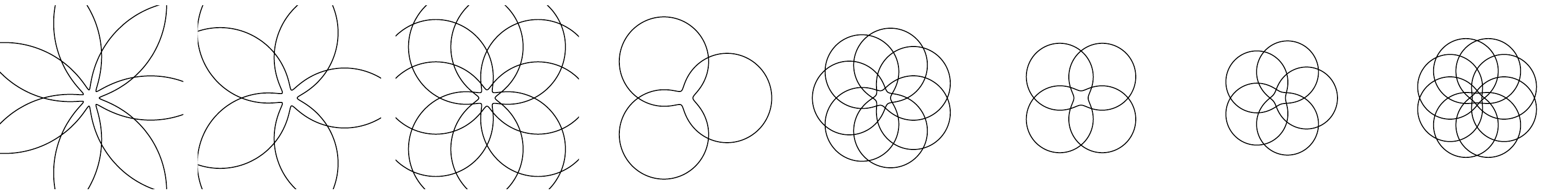} }
\caption{Periodic orbits of the spherical pendulum in stereographic projection of the sphere fixing the unstable equilibrium point 
with initial conditions along the circle in $(j_1, j_2)$ space with radius $0.75$ with rotation numbers
$4/7, 3/5, 5/8, 2/3, 5/7, 3/4, 4/5, 7/8$. Individual pictures are shown at the same scale cropped to $[-8, 8]^2$.}
\label{Wrat}
\end{figure}

Another quantity that is also interesting is the twist ${\cal T}$ given by the derivative of the rotation number
with respect to $j_2$ for constant $h$. The zeroes of this function indicate vanishing
twist, see \cite{DVN03}. Using the general formula derived there for the spherical pendulum we find
\[
         {2 \pi \cal T}(j_1, j_2) = 2 \pi ( - A W_1 + W_2) =  - \frac{j_1}{ |\hat j|^2} - \frac{9}{16} - \frac{3j_2^2}{4 |\hat j|^2 } + \frac{3}{8} \log \frac{32}{|\hat j|} + O(j_1, j_2)
         \,.
\]
where $W_i = \partial W(j_1, j_2) /\partial j_i$.
In polar coordinates the twist becomes
\[
  2\pi {\cal T} = - \frac{\sin s}{r} + \frac{3}{8} \ln \frac{32}{r} -\frac{15}{16} - \frac{3}{8} \cos 2s + O(r) \,.
\]
This shows that for $r \to 0$ to leading order the twist vanishes along the line $s = 0$, along which the rotation number 
approaches $3/4$. Considering the next order correction one can see that for $j_1 > 0$ only 
rotation numbers in the interval $(3/4, 1]$ occur.
The equation ${\cal T} = 0$ can be approximately solved for $s$.
This shows that the curve of twistless tori curves upward from the horizontal axis $j_1 = 0$.
The parametric form in polar coordinates shows that at the origin the twistless curve has 
slope zero, and that higher order derivatives do not exist.

Evaluating the rotation number on the curve of vanishing twist and expanding for small $r$ gives
\[
   W^* \approx  \frac{3}{4} + \frac{3r}{8 \pi} (  \ln \frac{32}{r} - \frac52 )  \,.
\]

\begin{figure}
\centerline{ \includegraphics[width=12cm]{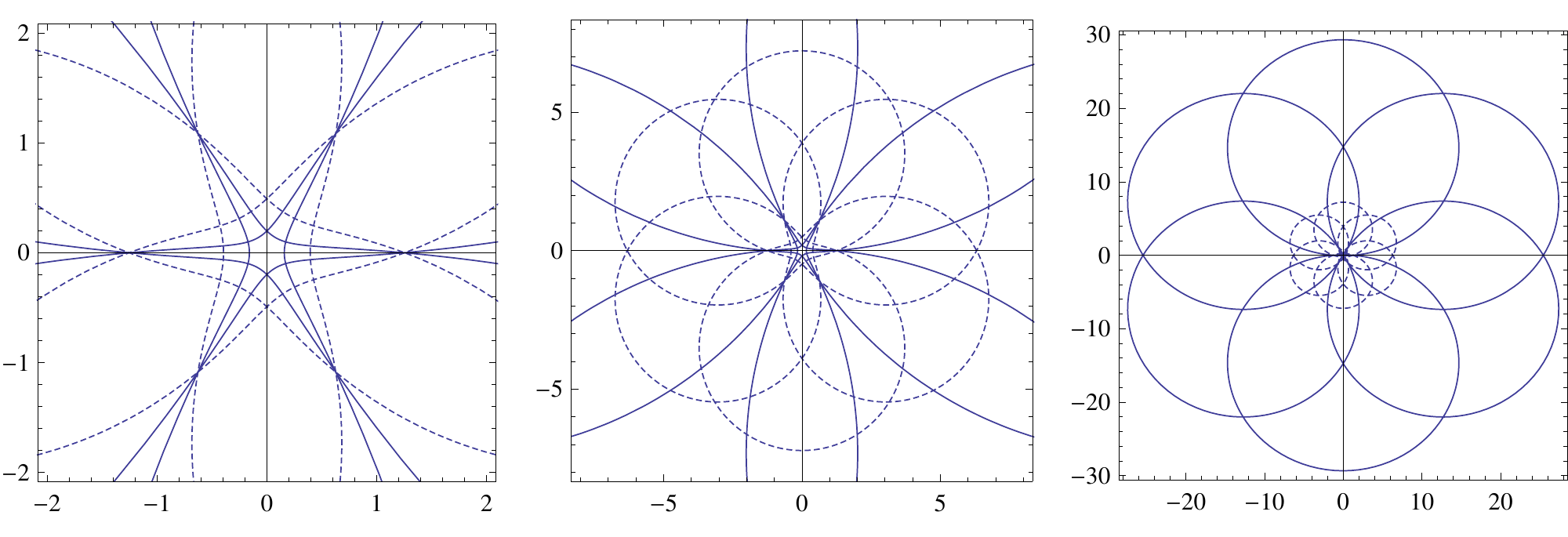} }
\caption{Two distinct periodic orbits of the spherical pendulum with same energy $h = 0.05$ and
same rotation number $W = 5/6$ shown in stereographic projection of the sphere fixing the unstable equilibrium point
in three different scales.
Existence of two orbit with the same rotation number for fixed energy implies that the twist 
vanishes for some initial condition ``in between''.
}
\label{WNT}
\end{figure}

\section{The Pendulum} \label{pendulum}

The pendulum Hamiltonian is obtained from the spherical pendulum Hamiltonian 
\eqref{eqn:Ham} by setting $j_2 = 0$:
\footnote{For the spherical pendulum we have chosen $h=0$ at the 
unstable equilibrium point which corresponds to $\theta= 0$. 
For the pendulum other conventions may seem more natural, 
but we stick with this unusual choice for ease of comparison with the spherical pendulum.}
\[
H =\frac12 \kappa \nu p_\theta^2 + \frac{\nu}{\kappa} (\cos\theta - 1)
    = \frac12 \kappa\nu  p_\theta^2 - 2 \frac{\nu}{\kappa}  \sin^2 \frac{\theta}{2}
 \,.
\]
Introducing the scaled energy and action as before 
the canonical action $I = I_1(h, j_2 = 0)$, the imaginary action $J = J_1(h, j_2 = 0)$ 
(which does have the interpretation of twice the tunnelling integral in this case), 
and their $h$-derivatives $T$ and $U$, the real and imaginary period, respectively, 
for $h > 0$ are given by
\begin{align*}
    I_+(h) & = \frac{1}{2\pi} \oint_\cyclereal \sqrt{2(h + 2 \sin^2\theta/2)}\, \dee \theta
           = \frac{4}{\pi k } E(k), \quad \quad k^2 = \frac{2}{2 + h} \\
      J_+(h) & = \frac{1}{2\pi} \oint_\cycleimag \sqrt{2(h - 2 \sinh^2 \theta/2)} \dee \theta
           = \frac{8}{\pi k} (K(k') - E(k')) \\
           T_+(h) & = 2\pi \frac{\partial I(h)}{\partial h} =2 k K(k) \\
        U_+(h) & = 2\pi \frac{\partial J(h)}{\partial h} = 4 k K( k' ), 
        \qquad k'^2 = 1 - k^2 = \frac{h}{2 + h} 
\end{align*}
For $h < 0$ the corresponding results may look unusual because they are not 
the true action or period of the pendulum, but the analogue of the action and 
period for the spherical pendulum, which differ by a factor $1/2$ 
(and thus the action becomes continuous across $h=0$). 
The results are
\begin{align*}
   I_-(h) & = \frac{4}{\pi}( E(k) - (1 - k^2) K(k)), \quad k^2 = \frac{2 + h}{2} \\
   J_-(h) & = -2 I_-(-2-h) =  \frac{8}{\pi}( k^2 K(k') -  E(k')) \\
   T_-(h) & = 2 K(k) \\
   U_-(h) & = 2 T_-( -2 - h)/2 \pi  = 4 K(k') , \qquad k'^2 = -\frac{h}{2} > 0
\end{align*}
The simple relation between $2 T_-$ and $U_-$ given by swapping $k$ for $k'$ 
was first noted by Appell \cite{Appell1879}. 
His interpretation of the imaginary period $U$ is that it is obtained by replacing time $t$ 
in the equations of motion by $\imag t$. 
We see that a similar relation holds for the action $2 I$ and the imaginary action $J$, 
except for another overall minus sign which is chosen such that $j = h + O(h^2)$
for small $h$, as required by the Birkhoff normal form near the hyperbolic equilibrium.
It should be noted, however, that Appel's relations are only true for energies below the critical energy $h=0$.  
The general question of the physical meaning of the action or period integrals over the 
$\cycleimag$ and $\cyclereal$ cycles in post quantum mechanics times is answered by uniform WKB quantisation, 
in which the action and the imaginary action both appear. 
In this paper we add another interpretation which is the relation between the Birkhoff normal form 
at an unstable equilibrium point and action integrals of the $\cycleimag$-cycles that vanish 
at this equilibrium point.

The series expansions for both signs of $h$ are
\begin{align*}
2 \pi I & = 8 + h +  \left( h - \frac{1}{16} h^2 + \dots \right) \ln(32/|h|) + \frac{3}{32} h^2 + \dots \\
        J & = h - \frac{1}{16} h^2 + \dots \\
        T & = \left ( 1 - \frac18 h - \frac{9}{256} h^2 +  \dots \right ) \ln( 32/|h|) + \frac14 h + \dots \\
        U/2\pi & = 1 - \frac18 h + \frac{9}{256} h^2 + \dots
\end{align*}

{\bf Remarks:} 
1) The four integrals satisfy the Legendre relation $ I U - J T= 8 $.
2)~Another way to interpret the actions given above for the case that $h < 0$ 
is to consider the pendulum 
reduced by the discrete symmetry $(\theta - \pi, p) \to (-\theta + \pi, -p)$.
Alternatively certain integer factors need to be introduced to recover the true 
semi-global symplectic invariants of the pendulum. See \cite{DVN05} for a discussion
of these integers in the case of the C.~Neumann system.
3) The two branches of $J$ and $U$ for $h< 0$ and $h>0$ join analytically at $h=0$.

For the pendulum the calculations are simpler because they do not involve elliptic integrals of the third kind.
In fact $J(h)$ and $I(h)$ are of  the second kind and are simply the integrals over the 
$\cycleimag$- and $\cyclereal$-cycle of the same elliptic curve. 
The semi-global symplectic invariants 
are obtained by computing $I \circ J^{-1}$, i.e.\ the integral over the $\cyclereal$-cycle
as a function of the integral over the $\cycleimag$-cycle. 
The result is 
\[
  2\pi I(j) = 8 + j( 1 + \ln ( 32/|j|) ) + \frac{3}{32} j^2 + \dots
\]
This can be done efficiently to high order because the series expansions of 
the complete elliptic integrals $K$ and $E$ are well known.
Hence we find the quadratic and higher order terms of the semi-global symplectic invariant of the pendulum are given by
\[
\frac{3 j^2}{32}
-\frac{5 j^3}{512}
+\frac{55 j^4}{32768}
-\frac{189 j^5}{524288}
+\frac{3689 j^6}{41943040}
-\frac{3129 j^7}{134217728}
 +\frac{1575405 j^8}{240518168576}...
\,.
\]
This can be obtained from the semi-global symplectic invariant of the spherical pendulum 
Theorem~\ref{theo:invariants} by setting $j_2 = 0$.
%

We are now going to connect the semi-global symplectic invariant 
to the Riemann matrix and $\vartheta$ functions. 
The main observation is that if we differentiate 
$I(J)$ with respect to $J$ we obtain the ratio $I'/J'$ where the prime 
denotes differentiation with respect to the energy.
Now $I'$ and $J'$ are the real and complex period, 
so that this gives the period ratio of the curve (up to a factor of $\imag$).
\footnote{When we speak of period ratio in this section we mean period 
ratio $\tau$ of the elliptic curve instead of the rotation number $W$.}
\[
 \frac{\partial I}{\partial J} = \frac{ \partial I/\partial h}{\partial J/ \partial h} = 
 \frac{T}{U} =
  \frac{ K(k) }{ 2 K(k') } = \frac{\imag \tau'}{2}
  =\frac{1}{2\pi} \left(  \ln \frac{32}{|j|} + \frac{ \frac14 j + O(j^2)} { 1 - \frac18 j + O(j^2) } \right) > 0
\]
where $\tau'$ denotes the reciprocal of the usual period ratio $\tau = \imag K'/K$,
since we are expanding in $h$, hence in $k'^2$ instead of in $k^2$. 

Notice that by passing from $I(J)$ to its derivative we do not 
lose any information since the constant term is not part of the 
semi-global symplectic invariant. 

In more degrees of freedom and near completely hyperbolic points 
this may give a variant of the Riemann matrix, 
e.g.\ in the Neumann system \cite{DVN05}. 
The pre-potential computed in \cite{DVN05} is a function of imaginary actions
whose Hessian is the Riemann matrix. In particular in this higher dimensional 
situation it may be advantageous to work with the Riemann matrix instead of with $I(J)$ 
as invariants).

Instead of the period ratio $\tau$ (or the Riemann matrix) commonly 
the so called nome $q = e^{2 \pi \imag \tau}$ is used.
Passing to the nome converts the logarithmic singularity in the period $T$ into a pole.
Exponentiation thus leads to a power series in $j$
\footnote{The sign in the exponent is fixed by convention so that $\Im(\tau) > 0$ 
which implies that the Riemann matrix has positive definite real part, so that
$q = e^{-R} = e^{\imag \pi\tau} = e^{-\pi K'/K} < 1$.
}
\begin{align}
   q & = \exp\left( -2 \pi \frac{\partial I }{ \partial J}\right) \nonumber \\
   & =  l - 6 l^2 + 48 l^2 - 436 l^4 + 4254 l^5 - 43452 l^6 \label{eqn:nomePend}
       + 458192 l^7 
        \pm \dots
\end{align}
where $l = j/32$. Alternatively we may consider the reciprocal which still 
`remembers' that there was a singularity which is now turned into a pole
\[
   \exp\left( 2 \pi \frac{\partial I }{ \partial J}\right) = 
       \frac{1}{l} + 6 - 12 l + 76 l^2 - 606 l^3 + 5412 l^4 \pm \dots
\]
The coefficients of either series are an equivalent set of ``modified'' semi-global symplectic invariants. 
Hence it appears that the nome itself (or the Riemann matrix in general) 
gives an equivalent set of modified semi-global symplectic invariants, 
when expressed as a function of the imaginary action $J$.
This is a very natural object to consider from the complex analytic point of view.
The expansion of the nome in terms of the modulus $k^2$ is well known
(the inverse of which can be expressed in terms of theta functions).
Now $k^2 = 2/(h + 2)$ for the pendulum and the Birkhoff normal form allows us to replace $h$ by $j$ and thus gives the 
final modified semi-global symplectic invariants.
Thus we can compute $q(J)$ from known functions and the Birkhoff normal form (whose functional 
inverse is also expressed in terms of known functions).

It turns out that even though we cannot find $q(J)$ directly, 
its inverse function $J(q)$ can be expressed in terms of extremely fast converging theta series 
$\vartheta_4 = \sum (-1)^n q^{n^2}$.
By standard identities involving thetanullwerte \cite{AS} we find that 
\begin{align*}
 J(q) 
     & = -16 q \frac{ \partial \vartheta_4/ \partial q}{ \vartheta_4^3}
     = 8q  \frac{\partial}{\partial q} \frac{1}{\vartheta^2_4}
      = 8 q \frac{\partial}{\partial q} \frac{\pi}{2K(q)} \\
      & = 
    32 \frac{q - 4 q^4 + 9 q^9 - 16 q^{16} + O(q^{25})}{(1 - 2q + 2 q^4 - 2 q^9 + 2 q^{16} + O(q^{25}) )^3 } \\
    &  = 32 ( q + 6 q^2 + 24 q^3 + 76 q^4 + \dots) 
\end{align*}
and the inverse of this series gives the modified semi-global symplectic invariants.
The coefficients of the $q$-series of $\pi/(2K)$ and of $J/32$ have been studied by 
J.W.L.\ Glaisher \cite{Glaisher1886}.
He did, however, not consider the inverse of this series. 
Since the formal inversion of a power series is an invertible process, one may consider taking
 $J(q)$ itself as the invariant instead of its inverse. However, this seems to far removed from the original 
construction, in which $J$ provides a local coordinate system near the singularity.

The above formula for $J(q)$ incidentally also proves that the modified semi-global symplectic invariants
are integers after scaling by 32, since division by a series with constant 
coefficient 1 and integer coefficients gives a series of the same form,
and similarly for inversion of a series with linear coefficient 1.
The fact that these coefficients are integers may be taken as another indicator why it may be a good idea to consider
the series $\exp( -2 \pi \partial I/\partial J)$ instead of the analytic part of $I(J)$
as the semi-global symplectic invariants.

Let us finally comment on how similar expressions can be obtained for 
the spherical pendulum. The key idea is to pass to the derivatives 
of the invariant, namely 
$\tau_1 = \partial I_1 / \partial j_1$ and 
$\tau_2 = \partial I_1 / \partial j_2$, which are entries of the period lattice, 
see \cite{VuNgoc03}.
Now observe that the singular terms are 
$\tau_1 = -\ln |\hat j| + \dots $ and 
$\tau_2 = \arg \hat j + \dots$.
Thus 
$\exp( -2\pi \tau_1) = | \hat j | + \dots$ and 
$\exp( 2 \pi \imag \tau_2) = \hat j / |\hat j| + \dots$ so that 
the combination $\hat q = \exp( -2\pi \tau_1 + 2\pi \imag \tau_2)$ is free 
of singularities and is a complex analogue of the nome $q$ in the case of the pendulum.
Repeating this construction with the full expression for $I_1$ in \eqref{eqn:Action}
including the symplectic invariant gives 
\begin{align*}
    \hat q & = \hat j \exp( -S_1 + \imag S_2)  \\
        & = l ( 1 + 6 ( l - 2 \bar l) - 3( 17l^2 + 2 l \bar l - 35 \bar l^2) + 2 (37 l^3 + 666 l^2 \bar l - 633 l \bar l -288\bar l^3 ) ...)
\end{align*}
where $l = \hat j/32 = (j_1 + \imag j_2)/32$. For $\bar l = l $ this reduces back to \eqref{eqn:nomePend} for the ordinary pendulum.
The construction can be done with $-S_1 - \imag S_2$ as well.
The main question that remains is whether the {\em inverse} of this series giving 
$j_1 + \imag j_2$ in terms of $\hat q$ and its conjugate has a simple expression in terms of 
say $\vartheta$-functions. This is harder to do than for the pendulum because 
the integral $J_1$ is of third kind.
But expressions for complete elliptic integrals of the third
kind in terms of the nome do not seem to be readily available, 
and would also need to depend on a second parameter, which would 
presumably be the imaginary part of $\hat q$).

\section*{Acknowledgement}

The author would like to thank George Papadopoulos for carefully reading the manuscript.
This research was supported by ARC grant DP110102001.

\goodbreak
\appendix

\section{Simultaneous linear normal form}
\label{app:Linear}

The symplectic linear normal forms near equilibria are due to Wiliamson \cite{Williamson36}.
Let the original variables be  $(\xi,p_\xi,\eta,p_\eta)$ with symplectic form 
$\Omega = \dee \xi \wedge \dee p_\xi + \dee \eta \wedge \dee p_\eta$, and assume that the equilibrium 
point is at the origin. 
The Hessian of the Hamiltonian at the origin is denoted by $D^2H$.
Let $J_4$ be the standard symplectic matrix corresponding 
to $\Omega$. For a linear focus-focus point the eigenvalues of
$J_4 D^2H$ are $\lambda = \pm \cycleimag \pm i \omega$,
 where we assume $\cycleimag \omega \not =  0, \omega > 0$.
The stable eigenspace is spanned by the eigenvectors
$v_s, \bar v_s$ of eigenvalues $\lambda_s, \bar \lambda_s = \cycleimag \pm i \omega$, 
and the unstable eigenspace is spanned by $v_u, \bar v_u$ of eigenvalues 
$\lambda_u, \bar \lambda_u = - \cycleimag \mp i \omega$. In each eigenspace
the dynamics is that of a (contracting or expanding) focus point, hence the name.

According to the Williamson (i.e.\ linear)  classification
the spherical pendulum has a saddle-saddle equilibrium with degenerate 
real eigenvalue and trivial Jordan blocks. From the non-linear point of view of the 
foliation of the integrable system it is a focus-focus point, because we may form 
a linear combination of the constants of motion $H$ and $L_z$ without changing 
the foliation. 
This is why we treat the saddle-saddle case as a degenerate focus-focus case, as 
is appropriate from the non-linear point of view. We know that for the 
non-linear system the angular momentum $L_z$ is a constant of motion,
even though it does not appear in the quadratic part. 

In fact we want a little bit more than the Williamson normal form, since we want to normalise the Hamiltonian and 
the second integral simultaneously. Since they commute, they can be simultaneously 
normalised. In the case of the spherical pendulum the angular momentum 
$L_z$ is already in normal form $L_z =  \xi p_\eta - \eta p_\xi$.
Thus we seek a linear symplectic transformation $M$ from $(\xi, p_\xi, \eta, p_\eta)$ to $(q_1,p_1,q_2,p_2)$ 
that leaves $J_2 = \xi p_\eta - \eta p_\xi = q_1 p_2 - q_2 p_1$ invariant 
and transforms the Hamiltonian to $\cycleimag J_1 + \omega J_2$
where $J_1 = q_1p_1 + q_2 p_2$. The eigenvalues of the equilibrium at
the origin are $\lambda = \pm \cycleimag \pm i \omega$, which in our case 
simply gives $\lambda = \pm \nu$, since $\omega = 0$. Hence $M$ must satisfy
\[
M^t J_4 M = J_4, \quad
M^t D^2J_2 M = D^2 J_2, \quad
M^t D^2 H M = \cycleimag D^2 J_1 + \omega D^2 J_2\,.
\]
Note that since $\{J_1, J_2\} = 0$  also their Hessians $D^2J_1$ and $D^2J_2$ commute,
and hence they diagonalise in the same basis.
The transformation given in lemma~\ref{Williamson} satisfies these conditions.

\section{Canonical Perturbation Theory} \label{pert}

Instead of Lie-series perturbation theory the analogue of canonical perturbation theory
may be used, which we briefly describe in this appendix.
Considering $H_4$ as given in \eqref{H4is} as a series of exponential function
we now introduce a formal averaging operator. 
Given a function 
\[
   F = \sum_{m \in \mathbb{Z}} f_{m} e^{m \theta_1}
\]
we define 
\[
\langle F \rangle =  f_{0}, \quad
\{ F \} = F - \langle F \rangle \,.
\]
The ``oscillating'' part $\{ F \}$ can be integrated to give the generating 
function of a canonical transformation. The average $\langle F \rangle$ 
cannot be removed by a canonical transformation. But it is independent 
of $\theta_1$, and this is the desired outcome.
The procedure is analogous to ordinary perturbation theory in which 
$\theta_1$ actually is an angle, and the series is trigonometric instead. 
The formal analogy works on the level of the series expansion, while the
usual interpretation as ``averaging over the flow of the unperturbed system'' 
does not work directly in our case because this flow is unbounded.
This is so because $J_1$ is not an action, i.e.\ it does not have a periodic flow. 
Nevertheless, $J_1$ is a local integral near the equilibrium point, thus
the name ``integral-angle'' variables.

In the present case this procedure gives
\[
   \langle H_4 \rangle = \frac{1}{16}J_1^2 + \frac{3}{16} J_2^2
\]
and by integration 
\begin{align*}
S_1(J_1, J_2, \theta_1, \theta_2) & = \int \{ H_4 \} d\theta_1 \\
         & = 
             \frac{5}{128} J^4 e^{-4\theta_1} + \frac{1}{16} J_1 J^2 e^{-2\theta_1} 
                    - \frac{1}{16} J_1 e^{2\theta_1} - \frac{5}{128} e^{4\theta_1} \,.
\end{align*}
The mixed variables generating function of the canonical transformation is
\[
    S(\tilde J_1, \tilde J_2, \theta_1, \theta_2) = \tilde J_1 \theta_1 + \tilde J_2 \theta_2 + \eps S_1(\tilde J_1, \tilde J_2, \theta_1, \theta_2)
\]
where $\tilde J_i$ are the new momenta.
Inverting the transformations explicitly up to first order gives
\begin{align*}
  J_i & = \tilde J_i + \eps \partial S_1( \tilde J_1, \tilde J_2, \tilde \theta_1, \tilde \theta_2) / \partial \tilde \theta_i \\
  \theta_i & = \tilde \theta_i - \eps \partial  S_1( \tilde J_1, \tilde J_2, \tilde \theta_1, \tilde \theta_2) / \partial \tilde J_i
\end{align*}
This leaves the angular momentum $J_2$ unchanged, 
while $J_1$ is modified by $-\eps \{ H_4\}$. Both angles are transformed in a 
non-trivial way. 

\section{Rotation Number}

The rotation number can also be  computed directly from elliptic integrals, 
instead of from the normal form and the semi-global symplectic invariants. 
For completeness, and to show the consistency of the approaches here we give
the expansion of the rotation number in $h$ and $j_2$. We start with 
the expression in Legendre normal form, which is
\[
W(h,j_2) = -\frac{\partial}{\partial j_2} I_1 =  \frac{1}{2\pi} \oint_\cyclereal \frac{j_2}{(1-\zeta^2)w} \, \dee \zeta 
   = \frac{j_2}{\pi\sqrt{2 (\zeta_2 - \zeta_0)}} \left( \frac{\Pi(n_-,k)}{1-\zeta_0}  + \frac{ \Pi(n_+,k)}{1+\zeta_0}\right) \,.
\]
Similarly, the period of the reduced motion is 
\[
T(h,j_2) = 2\pi \frac{\partial}{\partial h} I_1 = \oint_\cyclereal \frac{1}{w} \, \dee \zeta  
= \frac{2 \sqrt{2}}{\sqrt{\zeta_2 - \zeta_0}} K(k) \,.
\]
Expanding the elliptic integrals in the limit $h, j_2 \to 0$ gives

\begin{lem}
The rotation number of the spherical pendulum near the focus-focus point has the expansion
\[
\begin{aligned}
2 \pi W & =  2 \pi \sgn j_2   -  \Arg( h + i j_2)  + \frac{3}{8} j_2 \left(1 - \frac{5}{16}h+ \frac{35}{256} (h^2 + j_2^2) + \dots \right) \ln \frac{32}{\varrho}  \\
  & - \frac{j_2}{8} \frac{5h^2 + 6j_2^2}{\varrho^2} +
          \frac{j_2h}{256} \frac{77h^4 + 174j_2^2h^2 + 93j_2^4}{\varrho^4} + \dots
\end{aligned}
\]
where $h$ is the difference in energy to the critical value, $j_2$ is the angular momentum, 
and $ \varrho^2 = j_2^2 + h^2$.
\end{lem}

The smooth function multiplying $\ln$ is given by the residue series of the integrand of $W$, 
in the same way as $J_1$ is the residue series of the integrand of $I_1$.
Taking this function, substituting (\ref{hofj}),
and expanding in $j_1$ and $j_2$ reproduces the function $A(j_1, j_2)$ given by (\ref{Aofj}) above.

\bibliographystyle{plain}

\begin{thebibliography}{10}

\bibitem{AS}
Milton Abramowitz and Irene~A. Stegun, editors.
\newblock {\em Handbook of mathematical functions with formulas, graphs, and
  mathematical tables}.
\newblock Dover Publications Inc., New York, 1992.
\newblock Reprint of the 1972 edition.

\bibitem{Appell1879}
Paul~Emile Appell.
\newblock Sur une interpr{\'e}tation des valeurs imaginaires du temps en
  m{\'e}canique.
\newblock {\em Compt. Rend.}, 87:1074--1077, 1879.

\bibitem{Arnold78}
V.~I. Arnold.
\newblock {\em Mathematical Methods of Classical Mechanics}.
\newblock Springer, Berlin, 1978.

\bibitem{Audin02}
Mich{\`e}le Audin.
\newblock Hamiltonian monodromy via {P}icard-{L}efschetz theory.
\newblock {\em Comm. Math. Phys.}, 229(3):459--489, 2002.

\bibitem{BeukersCushman02}
Frits Beukers and Richard Cushman.
\newblock The complex geometry of the spherical pendulum.
\newblock In {\em Celestial mechanics (Evanston, IL, 1999)}, volume 292 of {\em
  Contemp. Math.}, pages 47--70. Amer. Math. Soc., Providence, RI, 2002.

\bibitem{Bolsinov95}
A.~V. Bolsinov.
\newblock Smooth orbital classification of integrable {H}amiltonian systems
  with two degrees of freedom.
\newblock {\em Mat. Sb.}, 186(1):3--28, 1995.

\bibitem{BD96}
A.~V. Bolsinov and H.~R. Dullin.
\newblock On the {E}uler case in rigid body dynamics and the {J}acobi problem
  (in {R}ussian).
\newblock {\em Regul. Chaotic Dyn.}, 2:13--25, 1997.

\bibitem{BolFom04}
A.~V. Bolsinov and A.~T. Fomenko.
\newblock {\em Integrable Hamiltonian Systems. Geometry, Topology,
  Classification}.
\newblock Chapman \& Hall/CRC, London, 2004.

\bibitem{BF71}
P.~F. Byrd and M.~D. Friedman.
\newblock {\em Handbook of Elliptic Integrals for Engineers and Physicists}.
\newblock Springer, Berlin, 1971.

\bibitem{Child98}
M.~S. Child.
\newblock Quantum states in a champagne bottle.
\newblock {\em J. Phys. A}, 31(2):657--670, 1998.

\bibitem{CushmanBates97}
R.~H. Cushman and L.~M. Bates.
\newblock {\em Global aspects of classical integrable systems}.
\newblock Birkh{\"a}user Verlag, Basel, 1997.

\bibitem{CushDuist88}
R.~H. Cushman and J.~J. Duistermaat.
\newblock The quantum mechanical spherical pendulum.
\newblock {\em Bull. Amer. Math. Soc.}, 19:475--479, 1988.

\bibitem{Molino94}
J.-P. Dufour, P.~Molino, and A.~Toulet.
\newblock Classification des syst\`emes int\'egrables en dimension {$2$} et
  invariants des mod\`eles de {F}omenko.
\newblock {\em C. R. Acad. Sci. Paris S\'er. I Math.}, 318(10):949--952, 1994.

\bibitem{Duistermaat80}
J.~J. Duistermaat.
\newblock On global action-angle coordinates.
\newblock {\em Comm. Pure Appl. Math.}, 33:687--706, 1980.

\bibitem{D03b}
H.~R. Dullin.
\newblock {P}oisson integrator for symmetric rigid bodies.
\newblock {\em Regul. Chaotic Dyn.}, 9:255--264, 2004.

\bibitem{DI03b}
H.~R. Dullin and A.~V. Ivanov.
\newblock Vanishing twist in the {H}amiltonian {H}opf bifurcation.
\newblock {\em Physica D}, 201:27--44, 2005.

\bibitem{DRVW99}
H.~R. Dullin, P.~H. Richter, A.~P. Veselov, and H.~Waalkens.
\newblock Actions of the {N}eumann system via {P}icard-{F}uchs equations.
\newblock {\em Physica D}, 155:159--183, 2001.

\bibitem{DVN03}
H.~R. Dullin and {V\~{u} Ng\d{o}c} S.
\newblock Vanishing twist near focus-focus points.
\newblock {\em Nonlinearity}, 17:1777--1785, 2004.

\bibitem{DVN05}
H.~R. Dullin and {V\~{u} Ng\d{o}c} S.
\newblock Symplectic invariants near hyperbolic-hyperbolic points.
\newblock {\em Regul. Chaotic Dyn.}, 12:689--716, 2007.

\bibitem{Eliasson84}
L.~H. Eliasson.
\newblock {\em Hamiltonian systems with Poisson commuting integrals}.
\newblock PhD thesis, University of Stockholm, 1984.

\bibitem{Gavrilov00}
Lubomir Gavrilov and Olivier Vivolo.
\newblock The real period function of {$A\sb 3$} singularity and perturbations
  of the spherical pendulum.
\newblock {\em Compositio Math.}, 123(2):167--184, 2000.

\bibitem{Glaisher1886}
J.W.L Glaisher.
\newblock On the coefficients in the q-series for {$\pi/2K$} and {$2G/\pi$}.
\newblock {\em Quart J. Pure and Applied Math.}, 21:60--76, 1886.

\bibitem{Horozov90}
E.~Horozov.
\newblock Perturbations of the spherical pendulum and {A}belian integrals.
\newblock {\em J. reine angew. Math.}, 408:114--135, 1990.

\bibitem{Horozov93}
E.~Horozov.
\newblock On the isoenergetical nondegeneracy of the spherical pendulum.
\newblock {\em Phys. Lett. A}, 173(3):279--283, 1993.

\bibitem{Ito89}
H.~Ito.
\newblock Convergence of {B}irkhoff normal forms for integrable systems.
\newblock {\em Comment. Math. Helv.}, 64(3):412--461, 1989.

\bibitem{Kruglikov99}
B.~S. Kruglikov.
\newblock Exact classification of divergence-free nondivergent vector fields on
  surfaces of small genus.
\newblock {\em Mat. Zametki}, 65(3):336--353, 1999.

\bibitem{MeyerHall92}
K.~R. Meyer and G.~R. Hall.
\newblock {\em Introduction to Hamiltonian Dynamical Systems and the N-Body
  Problem}.
\newblock Springer, Berlin, 1992.

\bibitem{Moser58}
J{\"u}rgen Moser.
\newblock On the generalization of a theorem of {A}. {L}iapounoff.
\newblock {\em Comm. Pure Appl. Math.}, 11:257--271, 1958.

\bibitem{RDWW96}
P.~H. Richter, H.~R. Dullin, H.~Waalkens, and J.~Wiersig.
\newblock Spherical pendulum, actions and spin.
\newblock {\em J. Phys. Chem.}, 100:19124--19135, 1996.

\bibitem{Ruessmann87}
H.~R{\"u}ssmann.
\newblock Nondegeneracy in the perturbation theory of integrable dynamical
  systems.
\newblock In {\em Number theory and dynamical systems (York, 1987)}, volume 134
  of {\em London Math. Soc. Lecture Note Ser.}, pages 5--18. Cambridge Univ.
  Press, Cambridge, 1989.

\bibitem{Vey78}
J.~Vey.
\newblock Sur certains syst\`emes dynamiques s\'eparables.
\newblock {\em Amer. J. Math.}, 100(3):591--614, 1978.

\bibitem{VuNgoc00}
San V{\~u}~Ng{\d{o}}c.
\newblock Bohr-sommerfeld conditions for integrable systems with critical
  manifolds of focus-focus type.
\newblock {\em Comm. Pure Appl. Math.}, 53(2):143--217, 2000.

\bibitem{VuNgoc03}
San V{\~u}~Ng{\d{o}}c.
\newblock On semi-global invariants for focus-focus singularities.
\newblock {\em Topology}, 42(2):365--380, 2003.

\bibitem{Whittaker37}
E.~T. Whittaker.
\newblock {\em A Treatise on the Analytical Dynamics of Particles and Rigid
  Bodies}.
\newblock Cambridge University Press, Cambridge, 4 edition, 1937.

\bibitem{Williamson36}
J.~Williamson.
\newblock On the algebraic problem concerning the normal forms of linear
  dynamical systems.
\newblock {\em Amer. J. of Math.}, 58(1):141--163, 1936.

\bibitem{Zung05}
Nguyen~Tien Zung.
\newblock Convergence versus integrability in {B}irkhoff normal form.
\newblock {\em Ann. of Math. (2)}, 161(1):141--156, 2005.

\end{thebibliography}

\def\cprime{$'$}

\end{document}